\documentclass{article}
\usepackage{comment}
\usepackage[utf8]{inputenc}
\usepackage[T1]{fontenc}
\usepackage[english]{babel}
\usepackage{amssymb}
\usepackage{amsmath, amsthm}
\usepackage{amsfonts}
\usepackage{graphics}
\usepackage{color}
\usepackage{xcolor}
\usepackage{graphicx}
\usepackage{graphics}
\usepackage{color}
\usepackage{amsmath}
\usepackage{amsfonts}
\usepackage{amssymb}
\usepackage{enumerate}
\usepackage{hyperref}
\usepackage{babel}

\newtheorem{Proposition}{Proposition}

\newtheorem{Theorem}{Theorem}

\newtheorem{Claim}{Claim}
\newtheorem{Lemma}{Lemma}
\newtheorem{Problem}{Problem}

 \title{Steiner connectivity problems in hypergraphs}
\author{Florian H\"orsch, Zolt\'an Szigeti}
\begin{document}
\maketitle
\begin{abstract}
We say that a tree $T$ is an $S$-Steiner tree if $S \subseteq V(T)$ and a hypergraph is an $S$-Steiner hypertree if it can be trimmed to an $S$-Steiner tree. We prove that it is NP-complete to decide, given a hypergraph $\mathcal{H}$ and some $S \subseteq V(\mathcal{H})$, whether there is a subhypergraph of $\mathcal{H}$ which is an $S$-Steiner hypertree. As corollaries, we give two negative results for two Steiner orientation problems in hypergraphs. Firstly, we show that it is NP-complete to decide, given a hypergraph $\mathcal{H}$, some $r \in V(\mathcal{H})$ and some $S \subseteq V(\mathcal{H})$, whether this hypergraph has an orientation in which every vertex of $S$ is reachable from $r$. Secondly, we show that it is NP-complete to decide, given a hypergraph $\mathcal{H}$ and some $S \subseteq V(\mathcal{H})$, whether this hypergraph has an orientation in which any two vertices in $S$ are mutually reachable from each other. This answers a longstanding open question of the Egerv\'ary Research group. We further show that it is NP-complete to decide if a given hypergraph has a well-balanced orientation.  On the positive side, we show that the problem of finding a Steiner hypertree and the first orientation problem can be solved in polynomial time if the number of terminals $|S|$ is fixed.
\end{abstract}
\section{Introduction}

This article is concerned with Steiner tree problems in hypergraphs and Steiner connectivity orientation problems in hypergraphs. Any undefined notation can be found in Section \ref{prel}.

The first part of the article deals with finding Steiner hypertrees in hypergraphs. There exists a rich literature on Steiner tree problems in graphs. For example, the problem of finding a Steiner tree minimizing a given weight function on the edges has been studied to a significant depth. It is well-known to be NP-complete \cite{karp}, several approximation results are known (\cite{KMB},\cite{BGRS}) and the problem is known to be fixed parameter tractable when parameterized by the number of terminals \cite{DW}. Another branch of research is concerned with the problem of packing Steiner trees. In particular, a famous conjecture of Kriesell \cite{K} remains open but several partial results are known (\cite{FKM},\cite{WW}). The corresponding algorithmic problem has been proven to be NP-complete by Kaski \cite{kaski}. On the other hand, we can trivially decide in polynomial time whether a given graph $G$ contains a single $S$-Steiner tree for some given $S \subseteq V(G)$. We here show that this situation drastically changes when considering hypergraphs. When $S=V(\mathcal{H})$ for a hypergraph $\mathcal{H}$, the problem can be solved in polynomial time using the concept of hypergraphic matroids which was introduced by Lor\'ea \cite{L} and exploited by Frank, Kir\'aly and Kriesell \cite{FKM}. Dealing with Steiner hypertrees in hypergraphs, we formally consider the following problem:
\medskip

\noindent \textbf{Steiner Hypertree (SHT):}
\smallskip

\noindent\textbf{Input:} A hypergraph $\mathcal{H}$, a set $S \subseteq V(\mathcal{H})$.
\smallskip

\noindent\textbf{Question:} Does $\mathcal{H}$ contain an $S$-Steiner hypertree?
\medskip

On the negative side, we show the following result:

\begin{Theorem}\label{treehard}
SHT is NP-complete.
\end{Theorem}

On the positive side, we are able to show that the problem can be solved in polynomial time if the number of terminals is fixed.

\begin{Theorem}\label{sfix}
There is a function $f:\mathbb{Z}_{\geq 0}\rightarrow \mathbb{Z}_{\geq 0}$ and an algorithm that solves SHT and runs in $O(f(|S|)n^{|S|}m^2)$.
\end{Theorem}
In the second part of this article, we apply these results to orientation problems in hypergraphs. We first deal with rooted connectivity. Formally, we consider the following problem:
\medskip

\noindent \textbf{Steiner Rooted Connected Orientation of Hypergraphs (SRCOH):}
\smallskip

\noindent\textbf{Input:} A hypergraph $\mathcal{H}$, a vertex $r \in V(\mathcal{H})$, a set $S \subseteq V(\mathcal{H})$.
\smallskip

\noindent\textbf{Question:} Is there an orientation $\vec{\mathcal{H}}$ of $\mathcal{H}$ that is $(r,S)$-Steiner rooted connected?
\medskip

It turns out that SHT and SRCOH are closely related. In particular, SRCOH can be solved in polynomial time when restricted to graphs. On the other hand, using the above mentioned relation, we prove the following result showing that such an algorithm is unlikely to exist for general hypergraphs.  

\begin{Theorem}\label{roothard}
SRCOH is NP-complete.
\end{Theorem}

Again exploiting this relation, we can conclude the following result from Theorem \ref{sfix}.

\begin{Theorem}\label{sfix2}
There is a function $f:\mathbb{Z}_{\geq 0}\rightarrow \mathbb{Z}_{\geq 0}$ and an algorithm that solves SRCOH and runs in $O(f(|S|)n^{|S|+1}m^2)$.
\end{Theorem}

We also deal with a more symmetric connectivity problem in orientations of hypergraphs. For graphs, a fundamental result of Nash-Williams \cite{NW} states that for any positive integer $k$, a graph has a $k$-arc-connected orientation if and only if it is $2k$-edge-connected. Actually, Nash-Williams proved  the even stronger result that every graph has a well-balanced orientation. In particular, this yields a complete characterization of the cases when a graph $G$ has an orientation $\vec{G}$ satisfying $\lambda_{\vec{G}}(u,v)\geq r(u,v)$ for some arbitrary symmetric requirement function $r:V(G) \times V(G) \rightarrow \mathbb{Z}_{\geq 0}$. For the case of global dyperedge-connectivity in hypergraphs, a characterization of the positive instances  has been proven with by Frank, Kir\'aly and Kir\'aly \cite{FKK}. 
 Their article does not provide a polynomial time algorithm to find the orientation in question if it exists. Such an  algorithm  by M\"uhlenthaler, Peyrille and Szigeti \cite{MPSz} is in preparation. 
The Egerv\'ary Research group \cite{egreswell} raised the question whether these approaches can be combined in order to find orientations satisfying local symmetric dyperedge-connectivity requirements of hypergraphs. We answer this question to the negative even for the very special case when the requirement function $r$ evaluates to 1 when both arguments belong to a fixed set of vertices and 0 otherwise. Formally, we consider the following problem:
\medskip

\noindent \textbf{Steiner Strongly Connected Orientation of Hypergraphs (SSCOH):}
\smallskip

\noindent\textbf{Input:} A hypergraph $\mathcal{H}$, a set $S \subseteq V(\mathcal{H})$.
\smallskip

\noindent\textbf{Question:} Is there an orientation $\vec{\mathcal{H}}$ of $\mathcal{H}$ that is strongly connected in $S$?
\medskip

As a rather simple consequence of Theorem \ref{roothard}, we are able to prove the following:

\begin{Theorem}\label{stronghard}
SSCOH is NP-complete.
\end{Theorem}

Finally, we deal with the problem of finding well-balanced orientations of hypergraphs. Recall that a celebrated theorem of Nash-Williams \cite{NW} states that every graph has a well-balanced orientation. It is easy to see that this result cannot be generalized to hypergraphs, but the complexity of deciding whether a given hypergraph has a well-balanced orientation is an open problem that was hinted at in \cite{egreswell} and asked explicitely in \cite{H}. Formally, we consider the following problem:
\medskip

\noindent \textbf{Well-balanced Orientation of Hypergraphs (WBOH):}
\smallskip

\noindent\textbf{Input:} A hypergraph $\mathcal{H}$.
\smallskip

\noindent\textbf{Question:} Is there a well-ballanced orientation $\vec{\mathcal{H}}$ of $\mathcal{H}$?
\medskip

We prove the following result that shows that a characterization of hypergraphs admitting a well-balanced orientation is unlikely to be found.

\begin{Theorem}\label{wellhard}
WBOH is NP-complete.
\end{Theorem}

The reduction proving Theorem \ref{wellhard} is similar to the ones proving Theorems \ref{treehard} and \ref{stronghard}, but slightly more involved.

In Section \ref{prel}, we provide some more formal definitions and preliminary results we need for our proofs. In Section \ref{redu}, we give the reductions proving Theorems \ref{treehard},\ref{roothard} and \ref{stronghard} and we prove Theorems \ref{sfix} and \ref{sfix2}.
\section{Preliminaries}\label{prel}
In this section, notation and some auxiliary results are collected. In Section \ref{def}, we give the necessary definitions and in Section \ref{res}, we give the preliminary results.
\subsection{Definitions}\label{def}
A {\it hypergraph} $\mathcal{H}$ consists of a vertex set $V(\mathcal{H})$ and a hyperedge set $\mathcal{E}(\mathcal{H})$ where each $e \in \mathcal{E}(\mathcal{H})$ is a subset of $V(\mathcal{H})$ of size at least 2. Throughout the article, we use $n$ and $m$ for the number of vertices and hyperedges of $\mathcal{H}$, respectively. If a hyperedge contains exactly two vertices $u$ and $v$, we call it an {\it edge} and write $uv$ instead of $\{u,v\}$. If each hyperedge in $\mathcal{E}(\mathcal{H})$ is an edge, we call $\mathcal{H}$ a graph. We say that a graph $G$ is a {\it trimming} of $\mathcal{H}$ if  $G$ is obtained from $\mathcal{H}$ by replacing every $e \in \mathcal{E}(\mathcal{H})$ by an edge containing two distinct vertices of $e$. For some  $X \subseteq V(\mathcal{H})$, we use $d_{\mathcal{H}}(X)$ for the number of hyperedges in $\mathcal{E}(\mathcal{H})$ that contain at least one vertex in $X$ and at least one vertex in $V(\mathcal{H})-X$. For a single vertex $v \in V(\mathcal{H})$, we use $d_{\mathcal{H}}(v)$ instead of $d_{\mathcal{H}}(\{v\})$ and call this value the {\it degree} of $v$ in ${\mathcal{H}}$. For $u,v \in V(\mathcal{H})$, we let $\lambda_{\mathcal{H}}(u,v)=\min\{d_\mathcal{H}(X):u \subseteq X \subseteq V(\mathcal{H})-v\}.$ For a non-negative integer $k$, a graph $G$ is called {\it $k$-edge-connected} if $d_G(X)\ge k$ for every nonempty $X\subsetneq V(G).$ A {\it tree} is an edge-minimal $1$-edge-connected graph.  For a positive integer $n$, the {\it number of labelled trees on $n$ vertices} refers to the number of trees whose vertex set is a fixed set $X$ of size $n$ where two trees $T_1,T_2$ are considered distinct if there is a pair of vertices $x_1,x_2 \in X$ such that exactly one of $E(T_1)$ and $E(T_2)$ contains an edge linking $x_1$ and $x_2$, even if $T_1$ and $T_2$ are isomorphic. A {\it $uv$-path} is a tree $P$ in which $d_P(u)=d_P(v)=1$ and $d_P(w)=2$ for all $w \in V(P)-\{u,v\}$ hold.  A {\it path} is a $uv$-path for some $u$ and $v.$ Given a tree $T$ and a {\it terminal set} $S \subseteq V(T)$, we say that $T$ is an {\it $S$-Steiner tree}. Two paths $P_1,P_2$ are called {\it internally vertex-disjoint} if $d_{P_1}(v)+d_{P_2}(v)\leq 2$ for all $v \in V(P_1)\cup V(P_2)$. An $S$-Steiner tree is called {\it small} if $|V(T)|\leq 2|S|-2$. A {\it subdivision} of a graph $G_1$ in a graph $G_2$ is a mapping $\phi:V(G_1)\rightarrow V(G_2)$ together with a collection of paths $\mathcal{P}=\{P_e:e \in E(G_1)\}$ such that for every $e=uv \in E(G_1)$, $P_e$ is a $\phi(u)\phi(v)$-path and the  paths in $\mathcal{P}$ are pairwise internally vertex-disjoint. For a graph $G$ and a vertex $v$ of $G$ which is contained in exactly two edges $uv$ and $vw$ such that $u \neq w$, we mean by {\it splitting off $v$} the operation which consists of deleting the vertex $v$ from $G$ and adding a new edge linking $u$ and $w$.

For a hypergraph $\mathcal{H}$, we denote by $G(\mathcal{H})$ the {\it incidence graph} of $\mathcal{H}$, i.e. the graph which is obtained from $\mathcal{H}$ by replacing every $e \in \mathcal{E}(\mathcal{H})$ by a new vertex $z_e$ and edges $vz_e$ for all $v \in e$. Given a terminal set $S$, a small $S$-Steiner tree $T$ and a hypergraph $\mathcal{H}$ with $S \subseteq V(\mathcal{H})$, a  subdivision $(\phi,\mathcal{P})$ of $T$ in $G(\mathcal{H})$ is called {\it special} if $\phi(s)=s$ for all $s \in S$ and $\phi(v)\in V(\mathcal{H})$ for all $v \in V(T)$.  An {\it $S$-Steiner hypertree} is a hypergraph that can be trimmed to an $S$-Steiner tree.

A {\it dypergraph} $\mathcal{D}$ consists of a vertex set $V(\mathcal{D})$ and a dyperedge set $\mathcal{A}(\mathcal{D})$ where each $a \in \mathcal{A}(\mathcal{D})$ is a tuple $(tail(a),head(a))$ where $head(a)$ is a vertex in $V(\mathcal{D})$ and $tail(a)$ is a nonempty subset of $V(\mathcal{D})-head(a)$.

For some $X \subseteq V(\mathcal{D})$, we say that a dyperedge $a\in\mathcal{A}(\mathcal{D})$ {\it enters} $X$ if $head(a)\in X$ and $tail(a)-X \neq \emptyset$. We denote by $\delta_{\mathcal{D}}^-(X)$ the set of dyperedges in $\mathcal{A}(\mathcal{D})$ that enter $X.$ We use $d_{\mathcal{D}}^-(X)$ for $|\delta_{\mathcal{D}}^-(X)|$.
 For $u,v \in V(\mathcal{D})$, we use $\lambda_{\mathcal{D}}(u,v)$ for $\min\{d_{\mathcal{D}}^-(X):X \subseteq V(\mathcal{D}), v \in X, u \in V(\mathcal{D})-X\}$.
 For some $S \subseteq V(\mathcal{D})$, we say that $\mathcal{D}$ is {\it strongly connected in $S$} if $\lambda_{\mathcal{D}}(u,v)\geq 1$ for every ordered  pair $(u,v)$ in $S$. We say that $\mathcal{D}$ is {\it strongly connected} if $\mathcal{D}$ is strongly connected in $V(\mathcal{D})$. For some $u,v \in V(\mathcal{D})$, we say that $v$ is {\it reachable} from $u$ if $\lambda_{\mathcal{D}}(u,v)\geq 1$. If for some $r\in V(\mathcal{D})$ and $S \subseteq V(\mathcal{D})$, every $v \in S$ is reachable from $r$, we say that $\mathcal{D}$ is {\it $(r,S)$-Steiner rooted connected}.
If a dypergraph $\vec{\mathcal{H}}$ is obtained from a hypergraph $\mathcal{H}$ by choosing a head for each hyperedge, we say that $\vec{\mathcal{H}}$ is an {\it orientation} of $\mathcal{H}$. We say that $\vec{\mathcal{H}}$ is a {\it well-balanced} orientation of $\mathcal{H}$ if $\lambda_{\vec{\mathcal{H}}}(u,v)\geq \lfloor\frac{1}{2}\lambda_{\mathcal{H}}(u,v)\rfloor$ holds for all ordered pairs $(u,v)$  in $V(\mathcal{H})$.
A dypergraph in which the tail of each dyperedge is of size 1 is called a {\it digraph}. The dyperedges of a digraph are called {\it arcs}. For some dypergraph $\mathcal{D}$, we let $D(\mathcal{D})$ denote the digraph in which every dyperedge $a\in \mathcal{A}(\mathcal{D})$ is replaced by a vertex $z_a$, an arc $vz_a$ for all $v \in tail(a)$ and an arc $z_a head(a)$.
The {\it underlying hypergraph} of $\mathcal{D}$ is the hypergraph on the same vertex set and that contains the hyperedge $tail(a)\cup head(a)$ for all $a \in \mathcal{A}(\mathcal{D})$. If $\mathcal{D}$ is a digraph, we speak of the {\it underlying graph.}
 For a non-negative integer $k$, a digraph $D$ is called {\it $k$-arc-connected} if $d_D^-(X)\ge k$ for every nonempty $X\subsetneq V(D).$
We say that a digraph $D$ is a {\it directed trimming} of $\mathcal{D}$ if  $D$ is obtained from $\mathcal{D}$ by replacing every $a \in \mathcal{A}(\mathcal{D})$ by an arc whose head is the head of $a$ and whose tail is a vertex in $tail(a)$. An $r$-arborescence is a digraph $B$ with $r \in V(B)$ and which is arc-minimal with the property that every vertex in $V(B)$ is reachable from $r$. For some $S \subseteq V(B)$, we speak of an {\it $(r,S)$-Steiner arborescence. A {\it directed $uv$-path} is a $u$-arborescence in which $v$ is the only vertex that is not the tail of any arc.}  A {\it circuit} is a strongly connected digraph satisfying $|A(D)|=|V(D)|$.
\medskip
\subsection{Preliminaries}\label{res}

For the reductions in Section \ref{redu}, we consider two variations of the well-known satisfiability problem. For a binary variable $x$, the {\it literals} over $x$ are $x$ and $\bar{x}$, the negation of $x$. The {\it literals over $X$} is the set of all literals over all $x \in X$. Let $\ell$ be a literal over some $x \in X$. Then $\bar{\ell}$ denotes $\bar{x}$ if $\ell=x$ and $x$ if $\ell=\bar{x}$. For an assignment $\phi:X \rightarrow \{TRUE,FALSE\}$, we say that $\phi(\ell)=TRUE$ if $\ell=x$ and $\phi(x)=TRUE$ or $\ell=\bar{x}$ and $\phi(x)=FALSE$, $\phi(\ell)=FALSE$, otherwise.
\medskip

\noindent \textbf{3SAT}
\smallskip

\noindent\textbf{Input:} A set of binary variables $X$, a set of clauses $\mathcal{C}$ each of which contains 3 literals over $X$.
\smallskip

\noindent\textbf{Question:} Is there an assignment $\phi:X \rightarrow \{TRUE,FALSE\}$ such that every clause of $\mathcal{C}$ contains at least one true literal?
\medskip

For the first reduction, we need the following well-known result, see \cite{karp}.

\begin{Theorem}\label{3satdure}
3SAT is NP-complete.
\end{Theorem}

We further consider the restricted problem $(3,B2)$-SAT which is obtained from 3SAT by restricting to instances $(X,\mathcal{C})$ in which for every $x \in X$, the literals $x$ and $\bar{x}$ appear exactly twice each.

We need the following strengthening of Theorem \ref{3satdure} that can be found in \cite{BKS}.

\begin{Theorem}\label{3b2satdure}
$(3,B2)$-SAT is NP-complete.
\end{Theorem}

For the second reduction, we need the following result that can be found in \cite{FKM}.
\begin{Proposition}\label{dirtrim}
Let $\mathcal{D}$ be a dypergraph, $r \in V(\mathcal{D})$ and $S \subseteq V(\mathcal{D})$. Then $\mathcal{D}$ contains a subdypergraph that can be transformed into an $(r,S)$-Steiner arborescence by a directed trimming if and only if all vertices in $S$ are reachable from $r$ in $\mathcal{D}$.
\end{Proposition}

We further need the following two results.

\begin{Proposition}\label{simple}
Let $\mathcal{D}$ be a dypergraph. Then for any pair $(u,v)$ of vertices in $V(\mathcal{D})$, we have $\lambda_{D(\mathcal{D})}(u,v)=\lambda_{\mathcal{D}}(u,v)$.
\end{Proposition}
\begin{proof}
First let $X \subseteq V(D(\mathcal{D}))$ with $v \in X, u \in  V(D(\mathcal{D}))-X$ and $d_{D(\mathcal{D})}^-(X)=\lambda_{D(\mathcal{D})}(u,v)$. Let $X'=X \cap V(\mathcal{D})$. Then for every $a\in \delta_{\mathcal{D}}^-(X')$, we have $head(a)\in X$ and $tail(a)-X \neq \emptyset$, so either the arc $z_ahead(a)$ enters $X$ in $D(\mathcal{D})$ or the arc $wz_a$ enters $X$ in $D(\mathcal{D})$ for some $w \in tail(a)$. Hence $\lambda_{\mathcal{D}}(u,v)\leq d_{\mathcal{D}}^-(X')\leq d_{D(\mathcal{D})}^-(X)=\lambda_{D(\mathcal{D})}(u,v)$.

Now let $X \subseteq V(\mathcal{D})$ with $v \in X, u \in  V(\mathcal{D})-X$ and $d_{\mathcal{D}}^-(X)=\lambda_{\mathcal{D}}(u,v)$. Let $X'\subseteq V(D(\mathcal{D}))$ be the set that contains $X$ and the vertex $z_a$ for all $a \in A(\mathcal{D})$ for which $tail(a)\subseteq X$ holds. Now every arc entering $X'$ in   $D(\mathcal{D})$ is of the form $z_a head(a)$ such that $a$ enters $X$ in $\mathcal{D}$.  Hence $\lambda_{D(\mathcal{D})}(u,v)\leq d_{D(\mathcal{D})}^-(X')\leq d_{\mathcal{D}}^-(X)=\lambda_{\mathcal{D}}(u,v)$.
\end{proof}
\begin{Proposition}\label{propch}
Let $\mathcal{H}$ be a hypergraph and let $e \in \mathcal{E}(\mathcal{H})$. Further, let ${\vec{\mathcal{H}}}_0$ be an orientation of $\mathcal{H}$ and let $\vec{e}_0$ be the orientation of $e$ in ${\vec{\mathcal{H}}}_0$. Suppose that there is some $x \in tail(\vec{e}_0)$ such that $\lambda_{\vec{\mathcal{H}}_0}(head(\vec{e}_0),x)\geq 1$. Then there is an orientation ${\vec{\mathcal{H}}}_1$ of $\mathcal{H}$ such that $\lambda_{\vec{\mathcal{H}}_1}(u,v)=\lambda_{\vec{\mathcal{H}}_0}(u,v)$ for every ordered pair  $(u,v)$  in $V(\mathcal{H})$ and $head(\vec{e}_1)=x$ where $\vec{e}_1$ is the orientation of $e$ in ${\vec{\mathcal{H}}}_1$.
\end{Proposition}

\begin{proof}

We obtain by Proposition \ref{simple} that $D(\vec{\mathcal{H}}_0)$ contains a directed path from $head(\vec{e}_0)$ to $x$. As this directed path contains none of the arcs $xz_{\vec{e}_0}$ and $z_{\vec{e}_0}head(\vec{e}_0)$, we obtain that $D(\vec{\mathcal{H}}_0)$ contains a circuit containing the arcs $xz_{\vec{e}_0}$ and $z_{\vec{e}_0}head(\vec{e}_0)$. Let $D_1$ be the digraph obtained from $D(\vec{\mathcal{H}}_0)$ by reversing all the arcs of this cycle. Note that we have $\lambda_{D_1}(u,v)=\lambda_{\vec{\mathcal{H}}_0}(u,v)$ for all ordered pairs $(u,v)$ in $V(\mathcal{H})$ and there is an orientation ${\vec{\mathcal{H}}}_1$ of $\mathcal{H}$ such that $D(\vec{\mathcal{H}}_1)=D_1$. Observe that $head(\vec{e}_1)=x$  where $\vec{e}_1$ is the orientation of $e$ in ${\vec{\mathcal{H}}}_1$. Further, for all ordered pairs  $(u,v)$ in $V(\mathcal{H})$, by Proposition \ref{simple}, we have $\lambda_{\vec{\mathcal{H}}_1}(u,v)=\lambda_{D(\vec{\mathcal{H}}_1)}(u,v)=\lambda_{D(\vec{\mathcal{H}}_0)}(u,v)=\lambda_{\vec{\mathcal{H}}_0}(u,v)$.
\end{proof}

For the proof of Theorem \ref{sfix}, we need the following result due to Kawarabayashi, Kobayashi and Reed \cite{KKR} which improves upon an earlier result of Robertson and Seymour \cite{rs}.

\begin{Lemma}\label{l4}
Let $G$ be a graph and $(u_1,v_1),$ $\ldots,$ $(u_k,v_k)$ pairs of vertices in $V(G)$. Then there  exist a function $f:\mathbb{Z}_{\geq 0}\rightarrow \mathbb{Z}_{\geq 0}$ and an algorithm that runs in $O(f(k)n^2)$ and decides whether there is a set of internally vertex-disjoint paths $P_1,\ldots,P_k$ such that $P_i$ is a $u_iv_i$-path for $i=1,\ldots,k$.
\end{Lemma}

We further need the following well-known property of trees.

\begin{Proposition}\label{deg3}
Let $T$ be a tree. Then the number of vertices $v \in V(T)$ with $d_T(v)=1$ is at least two more than the number of vertices  $v \in V(T)$ with $d_T(v)\geq 3$.
\end{Proposition}

We finally require the following classic theorem due to Cayley \cite{ca}.

\begin{Theorem}\label{cayley}
The number of distinct labelled trees on a ground set of $n$ vertices is $n^{n-2}$.
\end{Theorem}
\section{Proofs}\label{redu}

In this section, we give the proofs of the main theorems of this article.
\subsection{Steiner hypertrees}
This section is dedicated to proving the results on finding Steiner hypertrees in a given hypergraph. We first prove the negative result when the size of  the terminal set is not fixed in Section \ref{baumschwer} and then prove the positive result for a fixed number of terminals in Section \ref{fixleicht}.
\subsubsection{The proof of Theorem \ref{treehard}}\label{baumschwer}

\begin{proof}(of Theorem \ref{treehard})
 Clearly, the problem is in NP. We prove the hardness by a reduction from 3SAT. Let $(X,\mathcal{C})$ be an instance of 3SAT. 
 We now create an instance $(\mathcal{H},S)$ of SHT.
For every $x \in X$, we let $V(\mathcal{H})$ contain 2 vertices $w_x$ and $w_{\bar{x}}$. Next, for every $C \in \mathcal{C}$, we let $V(\mathcal{H})$ contain a vertex $z_C$. Further, we let $V(\mathcal{H})$ contain one more vertex $a$. Let $W=\bigcup_{x \in X}\{w_x,w_{\bar{x}}\}$ and $Z=\{z_C:C \in \mathcal{C}\}$.
For every $x \in X$, we let $\mathcal{E}(\mathcal{H})$ contain a hyperedge $e_x=\{a,w_x,w_{\bar{x}}\}$. Next, for every $C \in \mathcal{C}$, we let $\mathcal{E}(\mathcal{H})$ contain a hyperedge $e_C=\{\{w_\ell:\ell \in C\}\cup z_C\}$. 
Finally, we set $S= Z \cup a$.
This finishes the description of $(\mathcal{H},S)$. 

An illustration can be found in Figure \ref{fig0}.

\begin{figure}[h]
    \centering
        \includegraphics[width=.7\textwidth]{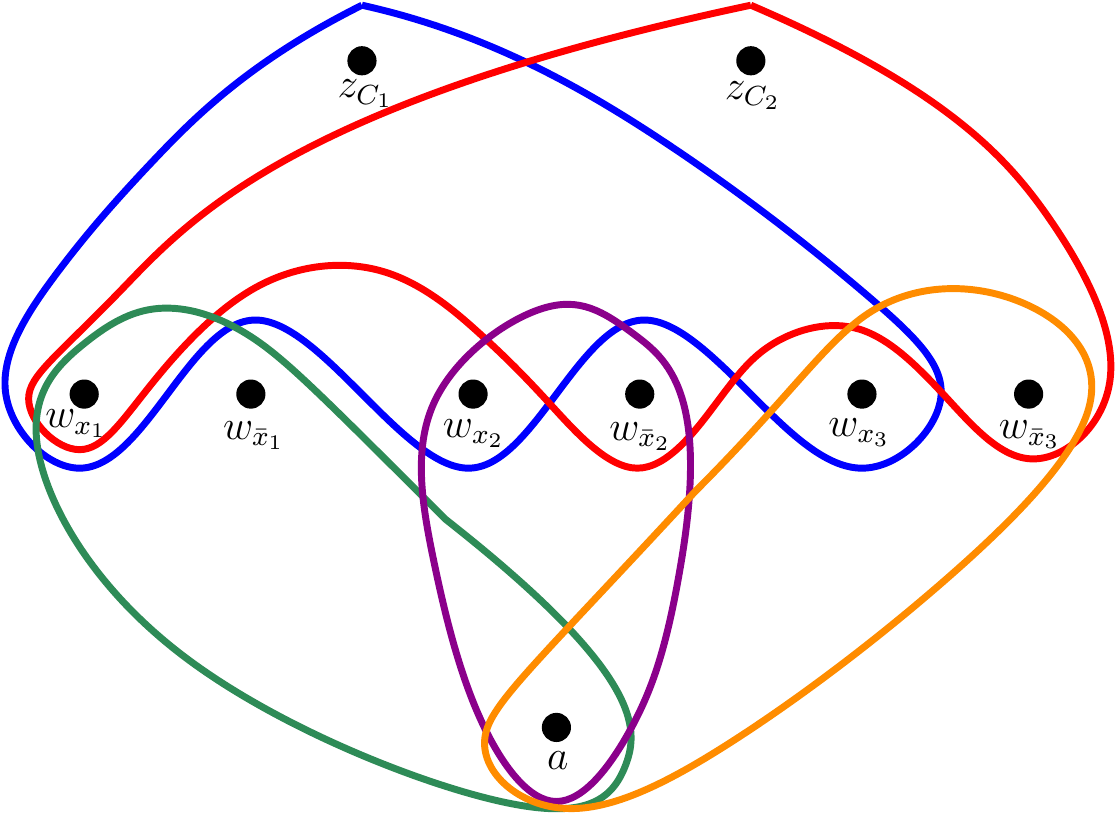}
        \caption{An example of the construction of $(\mathcal{H},S)$ for a 3SAT formula with $X=\{x_1,x_2,x_3\}$ and $\mathcal{C}=\{C_1=\{x_1,x_2,x_3\},C_2=\{x_1,\bar{x}_2,\bar{x}_3\}\}$.}\label{fig0}
\end{figure}
\medskip

We now prove that $(\mathcal{H},S)$ is a positive instance of SHT if and only if $(X,\mathcal{C})$ is a positive instance of 3SAT. 
\smallskip

First suppose that $(X,\mathcal{C})$ is a positive instance of 3SAT, so there is an assignment $\phi:X \rightarrow \{True,False\}$ that satisfies every clause of $\mathcal{C}$.

It suffices to prove that $\mathcal{H}$ can be trimmed to a $(Z \cup a)$-Steiner tree. For every $x \in X$, we trim $e_x$ to an edge $e'_x$ where $e'_x=aw_x$ if $\phi(x)=TRUE$ and $e'_x=aw_{\bar{x}}$ if $\phi(x)=FALSE$. Now consider some $C \in \mathcal{C}$. As $\phi$ is satisfying, we can choose some $\ell \in C$ such that $\phi(\ell)=TRUE$. We trim $e_C$ to an edge $e'_C$ where $e'_C=w_\ell z_C$. Now let $T$ be the graph that contains all vertices contained in $e'_x$ for some $x \in X$ and all the vertices contained in $e'_C$ for some $C \in \mathcal{C}$ and whose edge set is $\{e'_x:x \in X\}\cup \{e'_C:C\in \mathcal{C}\}$. Clearly, we have $Z \cup a \subseteq V(T)$. It also follows directly from the construction that $T$ is a tree. Hence $T$ is a $(Z \cup a)$-Steiner tree and so $\mathcal{H}$ is a $(Z \cup a)$-Steiner hypertree.

\medskip 

Now suppose that $(\mathcal{H},Z \cup a)$ is a positive instance of SHT, so there is a subhypergraph $\mathcal{T}$ of $\mathcal{H}$ which can be trimmed to a $(Z \cup a)$-Steiner tree $T$. We now define a truth assignment $\phi:X \rightarrow \{TRUE,FALSE\}$ in the following way: if the hyperedge $e_x$ is contained in $\mathcal{E}(\mathcal{T})$ and is trimmed to the edge $aw_x$ in $T$, we set $\phi(x)=TRUE$. Otherwise, we set $\phi(x)=FALSE$.

In order to prove that $\phi$ is satisfying, first observe that as $T$ is a $(Z \cup a)$-Steiner tree and as $z_C$ is contained in only one hyperedge of $\mathcal{E}(\mathcal{H})$, we obtain that $d_T(z_C)=1$ for all $C \in \mathcal{C}$. Now fix some $C^* \in \mathcal{C}$. Let $b$ be the unique vertex such that the edge $e_{C^*}$ is trimmed to $bz_{C^*}$ in $T$. As $T$ is a $(Z \cup a)$-Steiner tree, we obtain that $T$ contains a $z_{C^*}a$-path $P$. As $d_T(z_C)=1$ for all $C \in \mathcal{C}$, it follows that $V(P)\cap Z=z_{C^*}$. By construction, this yields that $P=abz_{C^*}$. It follows that the unique hyperedge in $\mathcal{E}(\mathcal{H})$ containing $a$ and $b$ is trimmed to $ab$ in $T$.  We obtain that $b=w_\ell$ for some $\ell \in C^*$ and $\phi(\ell)=TRUE$, so $C^*$ is satisfied by $\phi$.
 As $C^*$ was chosen arbitrarily, $\phi$ is a satisfying assignment for $(X,\mathcal{C})$.

As the size of $\mathcal{H}$ is clearly polynomial in  the size of $(X,\mathcal{C})$ and by Theorem \ref{3satdure}, the statement follows.
\end{proof}

\subsubsection{Polynomial algorithm for a fixed number of terminals}\label{fixleicht}

This section is dedicated to proving Theorem \ref{sfix}. We first show that in order to do so, it suffices to consider a related problem in the incidence graph of the given hypergraph.  Recall that, for a hypergraph $\mathcal{H}$,  the incidence graph $G(\mathcal{H})$  of $\mathcal{H}$ is the graph which is obtained from $\mathcal{H}$ by replacing every $e \in \mathcal{E}(\mathcal{H})$ by a new vertex $z_e$ and edges $vz_e$ for all $v \in e$.
\begin{Lemma}\label{kurz1}
Let $\mathcal{H}$ be a hypergraph and $S \subseteq V(\mathcal{H})$. Then $\mathcal{H}$ contains an $S$-Steiner hypertree if and only if $G(\mathcal{H})$ contains an $S$-Steiner tree $T$ that satisfies $d_T(z_e)=2$ for all $e \in \mathcal{E}(\mathcal{H})$ with $z_e \in V(T)$.
\end{Lemma} 
\begin{proof}
First suppose that $\mathcal{H}$ contains an $S$-Steiner hypertree $\mathcal{T}$ that can be trimmed to an $S$-Steiner tree $T$. Let $T'$ be obtained from $T$ by subdividing every edge $\tilde{e}\in E(T)$, creating the vertex $z_e$, where $e \in \mathcal{E}(\mathcal{H})$ is the edge from which $\tilde{e}$ is obtained by trimming. Observe that $T'$ is a subgraph of $G(\mathcal{H})$. By construction, we have $S \subseteq V(T)\subseteq V(T')$ and $d_{T'}(z_e)=2$ for all $e \in \mathcal{E}(\mathcal{H})$ with $z_e \in V(T)$. Finally, as $T$ is a tree, so is $T'$.

Now suppose that $G(\mathcal{H})$ contains an $S$-Steiner tree $T'$ that satisfies $d_{T'}(z_e)=2$ for all $e \in \mathcal{E}(\mathcal{H})$ with $z_e \in V(T')$. Let $T$ be the graph with $V(T)=V(T')\cap V(\mathcal{H})$ and which contains an edge $\tilde{e}=uv$ for all $u,v \in V(T)$ for which there is some $e \in \mathcal{E}(\mathcal{H})$ with $uz_e,vz_e \in E(T')$. Observe that $\tilde{e}$ can be obtained from $e$ by trimming and hence $T$ can be obtained from a subhypergraph $\mathcal{T}$ of $\mathcal{H}$ by trimming. As $T$ is obtained from $T'$ by contracting edges, we obtain that $T$ is a tree and by construction, we have $S \subseteq V(T')\cap V(\mathcal{H})=V(T)$. Hence $T$ is an $S$-Steiner tree and so $\mathcal{T}$ is an $S$-Steiner hypertree. 
\end{proof}

We next show that it suffices to deal with small $S$-Steiner trees instead of arbitrary ones which is important to limit the number of possible choices.

\begin{Lemma}\label{kurz2}
Let $\mathcal{H}$ be a hypergraph and $S \subseteq V(\mathcal{H})$. Then $G(\mathcal{H})$ contains an $S$-Steiner tree $T$ that satisfies $d_T(z_e)=2$ for all $e \in \mathcal{E}(\mathcal{H})$ with $z_e \in E(T)$ if and only if $G(\mathcal{H})$ contains a small $S$-Steiner tree as a special subdivision.
\end{Lemma}
\begin{proof}
First suppose that $G(\mathcal{H})$ contains a small $S$-Steiner tree $T$ as a special subdivision $(\phi,\mathcal{P})$. Let $T'$ be the subgraph of $G(\mathcal{H})$ with $V(T')=\bigcup_{P \in \mathcal{P}}V(P)$ and $E(T')=\bigcup_{P \in \mathcal{P}}E(P)$. As the paths of $\mathcal{P}$ are pairwise internally vertex-disjoint, we obtain that $T'$ can be obtained from $T$ by subdividing edges several times. Hence $T'$ is a tree. Further, we have $S \subseteq V(T)\subseteq V(T')$. Finally, as $V(T)\subseteq V(\mathcal{H})$ and the paths in $\mathcal{P}$ are pairwise internally vertex-disjoint, we obtain that $d_{T'}(z_e)=2$ for all $e \in \mathcal{E}(\mathcal{H})$ with $z_e \in E(T')$. Hence $T'$ is an $S$-Steiner tree that satisfies $d_{T'}(z_e)=2$ for all $e \in \mathcal{E}(\mathcal{H})$ with $z_e \in E(T')$.

Now suppose that $G(\mathcal{H})$ contains an $S$-Steiner tree $T'$ that satisfies $d_{T'}(z_e)=2$ for all $e \in \mathcal{E}(\mathcal{H})$ with $z_e \in V(T')$. Choosing $T'$ minimum, we may suppose that every vertex of degree $1$ of $T'$ is contained in $S$. Let $T$ be obtained from $T'$ by 
 splitting off vertices of degree $2$ which are not contained in $S$.  Observe that, as $d_{T'}(z_e)=2$ for all $e \in \mathcal{E}(\mathcal{H})$ with $z_e \in V(T')$, we  have $V(T)\subseteq V(\mathcal{H})$. Further, by construction, we have $d_T(v)\geq 3$ for all $v \in V(T)-S$. By Proposition \ref{deg3}, we obtain $|V(T)|\leq 2|S|-2$. Hence $T$ is a small $S$-Steiner tree. In order to see that $G(\mathcal{H})$ contains $T$ as a special subdivision, let $\phi$ be the identity map on $V(T)$. Further, for every $e=uv \in E(T)$, let $P_e$ be the unique $uv$-path in $T'$ and let $\mathcal{P}=\{P_e:e \in E(T)\}$. As $T'$ is a tree, the $P_e$ are pairwise internally vertex-disjoint and hence $(\phi,\mathcal{P})$ is a special subdivision of $T$ in $G(\mathcal{H})$.
\end{proof}
 We now show that a special subdivision of a fixed small $S$-Steiner tree can be found efficiently.
\begin{Lemma}\label{trouver}
Let $S$ be a set with $|S|=k$, $T$ a small $S$-Steiner tree and $\mathcal{H}$ a hypergraph. Then there exist a function $f:\mathbb{Z}_{\geq 0}\rightarrow \mathbb{Z}_{\geq 0}$ and an algorithm that  tests whether $G(\mathcal{H})$ contains $T$ as a special subdivision and runs in $O(f(k)n^{k}m^2)$.
\end{Lemma}
\begin{proof}
When trying to find a special subdivision $(\phi,\mathcal{P})$ of $T$ in $G(\mathcal{H})$, first observe that, as $T$ is small, there are at most $n^{|V(T)-S|}\leq n^{k-2}$ possibilities to choose $\phi$. We now fix some $\phi_0$. In order to test whether there exists a special subdivision $(\phi,\mathcal{P})$ of $T$ in $G(\mathcal{H})$ with $\phi=\phi_0$, it suffices to decide whether there exists a set of pairwise internally vertex-disjoint paths $\{P_e:e \in E(T)\}$ in $G(\mathcal{H})$ such that $P_e$ is a $uv$-path for every $e=uv \in E(T)$. By Lemma \ref{l4}, there is a function $f':\mathbb{Z}_{\geq 0}\rightarrow \mathbb{Z}_{\geq 0}$ that tests this property and runs in $O(f'(|E(T)|)(n+m)^2)$. We obtain a total running time of $O(n^{k-2}(f'(|E(T)|)(n+m)^2))=O(f'(2k)n^{k}m^2)$.
\end{proof}

We next prove that the number of small $S$-Steiner trees on a fixed ground set is bounded.

\begin{Lemma}\label{compter}
Let $S$ be a set with $|S|=k$  for some integer $k \geq 2$ and $X$ a set with $|X|=k-2$ and $X \cap S = \emptyset$. Then there are at most $(2k-2)^{2k-3}$ labelled  small Steiner trees $T$ with $V(T)\subseteq S \cup X$.
\end{Lemma}

\begin{proof}
By Theorem \ref{cayley}, the lemma follows from the fact that the number of labelled trees on at most $n$ vertices is at most  $\sum_{\mu=1}^n{n\choose\mu}\mu^{\mu-2}<\sum_{\mu=1}^n n^{n-\mu}\cdot n^{\mu-2}=\sum_{\mu=1}^n n^{n-2}=n^{n-1}.$
\end{proof}

We are now ready to prove Theorem \ref{sfix}.
\begin{proof}(of Theorem \ref{sfix})
Let $\mathcal{H}$ be a hypergraph and $S \subseteq V(\mathcal{H})$ with $|S|=k$. We need to decide in $O(f(k)n^{k}m^2)$ time whether $\mathcal{H}$ contains an $S$-Steiner hypertree.
By Lemmas \ref{kurz1} and \ref{kurz2}, it suffices to show that we can decide in $O(f(k)n^{k}m^2)$ time whether $G(\mathcal{H})$ contains a small $S$-Steiner tree as a special subdivision. In order to test whether this is the case, we can restrict ourselves to trees $T$ that satisfy $V(T)\subseteq S \cup X$ for some fixed $X$ with $|X|=k-2$. By Lemma \ref{compter}, there is a function $f_1:\mathbb{Z}_{\geq 0}\rightarrow \mathbb{Z}_{\geq 0}$ such that there are at most $f_1(k)$ such trees. Next, by Lemma \ref{trouver}, there is a function $f_2:\mathbb{Z}_{\geq 0}\rightarrow \mathbb{Z}_{\geq 0}$ such that, for each of these trees $T$, we can decide in $O(f_2(k)n^{k}m^2)$ time if $G(\mathcal{H})$ contains $T$ as a special subdivision. We obtain a total running time of $O(f_1(k)f_2(k)n^{k}m^2)$.
\end{proof}
\subsection{Steiner rooted-connected orientations}

This section is dedicated to proving Theorems \ref{roothard} and \ref{sfix2}. The following is the key ingredient. 
 Its graphic version is trivial and easily implies the hypergraphic one.

\begin{Lemma}\label{equi}
Let $\mathcal{H}$ be a hypergraph, $r \in V(\mathcal{H})$ and $S \subseteq V(\mathcal{H})$. Then $\mathcal{H}$ has a $(r,S)$-Steiner rooted connected orientation if and only if $\mathcal{H}$ contains a $(S \cup r)$-Steiner hypertree.
\end{Lemma}
\begin{proof}
First suppose that $\mathcal{H}$ contains an $(S \cup r)$-Steiner hypertree $\mathcal{T}$ that can be trimmed to an $(S \cup r)$-Steiner tree $T$. Then there is an orientation $\vec{T}$ of $T$ that is an $(r,S)$-Steiner arborescence. Now consider the orientation $\vec{\mathcal{H}}$ in which for every hyperedge in $\mathcal{E}(\mathcal{T})$ we choose as its head the head of the corresponding arc in $\vec{T}$ and assign an arbitrary orientation to all remaining hyperedges. Let $\vec{\mathcal{T}}$ be the orientation of $\mathcal{T}$ obtained from $\vec{\mathcal{H}}$ by a restriction to the dyperedges whose corresponding hyperedges are contained in $\mathcal{E}(\mathcal{T})$. Then $\vec{\mathcal{T}}$ can be transformed into $\vec{T}$ by a directed trimming. As $\vec{\mathcal{T}}$ is a subdypergraph of $\vec{\mathcal{H}}$ , we obtain by Proposition \ref{dirtrim} that all vertices in $S$ are reachable from $r$ in $\vec{\mathcal{H}}$.

Now suppose that $\mathcal{H}$ has an $(r,S)$-Steiner rooted connected orientation $\vec{\mathcal{H}}$. By Proposition \ref{dirtrim}, $\vec{\mathcal{H}}$ contains a subdypergraph $\vec{\mathcal{T}}$ that can be transformed into an $(r,S)$-Steiner arborescence $\vec{T}$ by a directed trimming. Let $\mathcal{T}$ be the underlying hypergraph of $\vec{\mathcal{T}}$ and $T$ the underlying graph of $\vec{T}$. Then $T$ is an $(S \cup r)$-Steiner tree, $T$ can be obtained from $\mathcal{T}$ by trimming and $\mathcal{T}$ is a subhypergraph of $\mathcal{H}$. This finishes the proof.
\end{proof}

Lemma \ref{equi} and Theorem \ref{treehard} imply Theorem \ref{roothard}. Further, Lemma \ref{equi} and Theorem \ref{sfix} imply Theorem \ref{sfix2}.

\subsection{Steiner strongly connected orientations}

We here conclude Theorem \ref{stronghard} from Theorem \ref{roothard}.
\begin{proof}(of Theorem \ref{stronghard})
Again, the problem clearly is in NP. We prove the hardness by a reduction from SRCOH. Let $(\mathcal{H},r,S)$ be an instance of SRCOH. Let $\mathcal{H}'$ be obtained from $\mathcal{H}$ by adding the hyperedge $e^*=S \cup r$ and let $S'=S\cup r$. We will prove that $(\mathcal{H}',S')$ is a positive instance of SSCOH if and only if  $(\mathcal{H},r,S)$ is a positive instance of SRCOH.

First suppose that $(\mathcal{H},r,S)$ is a positive instance of SRCOH, so there is an orientation $\vec{\mathcal{H}}$ of $\mathcal{H}$ in which all vertices of $S$ are reachable from $r$. Let the orientation $\vec{\mathcal{H}'}$ of $\mathcal{H}'$ be obtained by choosing $r$ as the head of $\vec{e^*}$ and giving all other hyperedges the orientation they have in $\vec{\mathcal{H}}$. As $\vec{\mathcal{H}}$ is a subdypergraph of $\vec{\mathcal{H}'}$, we obtain that all vertices in $S$ are reachable from $r$ in $\vec{\mathcal{H}'}$. Further, due to the orientation of $e^*$, $r$ is also reachable from all vertices in $S$ in $\vec{\mathcal{H}'}$. Hence $\vec{\mathcal{H}'}$ is strongly connected in $S'$.

Now suppose that $(\mathcal{H}',S')$ is a positive instance of SSCOH, so there is an orientation $\vec{\mathcal{H}'}$ of $\mathcal{H}'$ which is strongly connected in $S'$. By Proposition \ref{propch}, we may suppose that $r$ is the head of $\vec{e^*}$ in $\vec{\mathcal{H}'}$. Now let $\vec{\mathcal{H}}$ be the orientation of $\mathcal{H}$ which is obtained from $\vec{\mathcal{H}'}$ by deleting the dyperedge corresponding to $e^*$. Consider some $s \in S$. As $\vec{\mathcal{H}'}$ is strongly connected in $S$, there is a subdypergraph $\vec{\mathcal{T}}$ of $\vec{\mathcal{H}'}$ that can be trimmed to a directed $rs$-path. Clearly, this path does not contain an arc entering $r$ and hence $\vec{\mathcal{T}}$ does not contain the dyperedge corresponding to $e^*$. It follows that $\vec{\mathcal{T}}$ is also a subdypergraph of $\vec{\mathcal{H}}$. As $s$ was chosen arbitrarily, we obtain that all vertices in $S$ are reachable from $r$ in $\vec{\mathcal{H}}$.

As the size of $(\mathcal{H}',S')$ is clearly polynomial in the size of $(\mathcal{H},r,S)$ and by Theorem \ref{roothard}, the statement follows.

\end{proof}

\subsection{Well-balanced orientation}
\begin{proof}(of Theorem \ref{wellhard}) 
Again, the problem clearly is in NP. We prove the hardness by a reduction from $(3,B2)$-SAT. Let $(X,\mathcal{C})$ be an instance of $(3,B2)$-SAT. We now create an instance $\mathcal{H}$ of WBOH. For every literal $\ell$ over $X$ that is contained in two clauses $C_1,C_2 \in \mathcal{C}$, we let $V(\mathcal{H})$ contain a set $W_\ell=\{v_\ell,w_\ell,v_\ell^{C_1},v_\ell^{C_2},w_\ell^{C_1},w_\ell^{C_2}\}$ of vertices and we add edges to $\mathcal{E}(\mathcal{H})$ so that $v_\ell v_\ell^{C_1}w_\ell^{C_1}w_\ell w_\ell^{C_2} v_\ell^{C_2}v_\ell$ becomes a cycle. Next for every $C \in \mathcal{C}$ that contains 3 literals $\ell_1,\ell_2,\ell_3$, we let $V(\mathcal{H})$ contain a vertex $z_C$ and we let $\mathcal{E}(\mathcal{H})$ contain the edges $v_{\ell_i}^Cz_C$ for $i=1,2,3$ and the hyperedge $\{w_{\ell_1}^C,w_{\ell_2}^C,w_{\ell_3}^C,z_C\}$. We now let $W$ denote $\bigcup_{x \in X}W_x \cup W_{\bar{x}}$ and we let $Z$ denote $\bigcup_{C \in \mathcal{C}}z_C$. Finally, we let $V(\mathcal{H})$ contain one more vertex $a$ and we let $\mathcal{E}(\mathcal{H})$ contain an edge $av_\ell$ for all literals $\ell$ over $X$, a hyperedge $\{a,w_x,w_{\bar{x}}\}$ for all $x \in X$ and 4 copies $e_1^*,\ldots,e_4^*$ of the hyperedge $Z \cup a$. This finishes the description of $\mathcal{H}$. For an illustration, see Figure \ref{figw1}.

\begin{figure}[h]
    \centering
        \includegraphics[width=\textwidth]{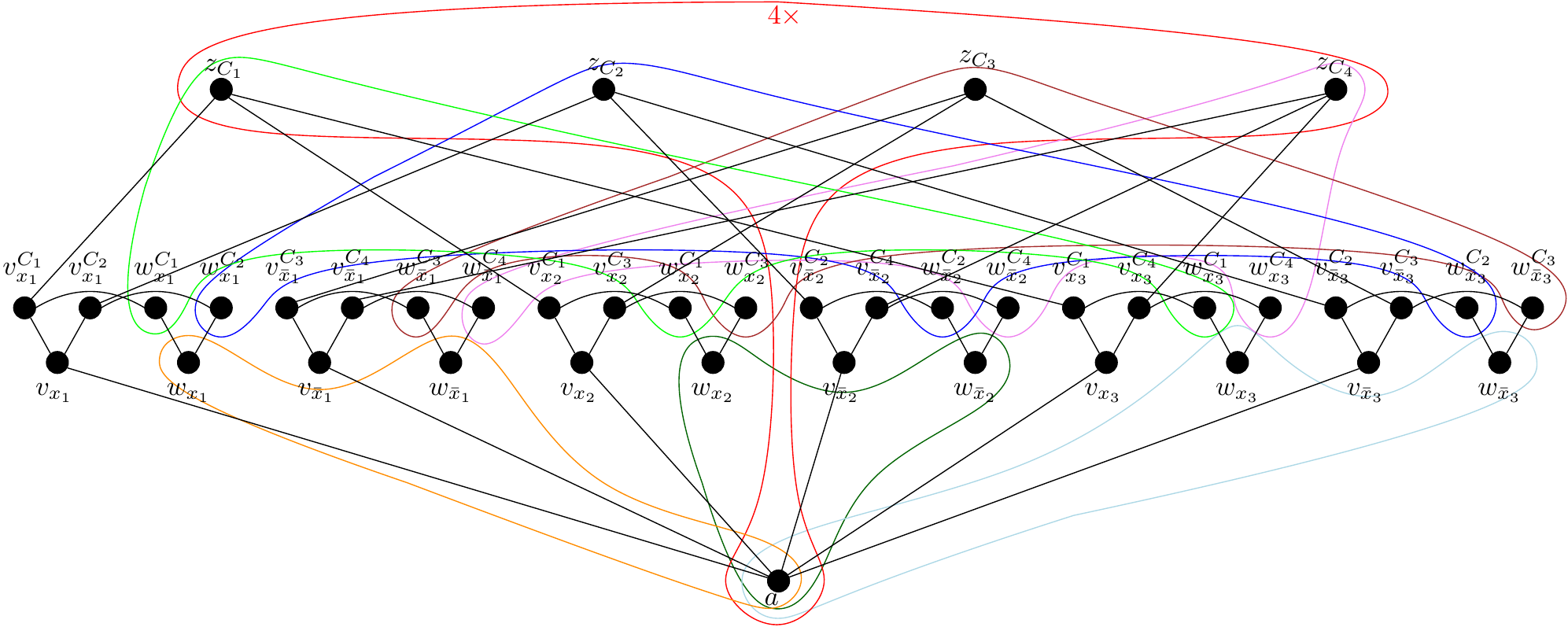}
        \caption{An example of the construction of $\mathcal{H}$ for a $(3,B2)$-formula with $X=\{x_1,x_2,x_3\}$ and $\mathcal{C}=\{C_1=\{x_1,x_2,x_3\},C_2=\{x_1,\bar{x}_2,\bar{x}_3\},C_3=\{\bar{x}_1,x_2,\bar{x}_3\},C_4=\{\bar{x}_1,\bar{x}_2,{x}_3\}\}$.}\label{figw1}
\end{figure}

We now show that $\mathcal{H}$ is a positive instance of WBOH if and only if $(X,\mathcal{C})$ is a positive instance of $(3,B2)$-SAT. First suppose that $\mathcal{H}$ is a positive instance of WBOH.   Before giving the main proof, we need the following auxiliary result.
\begin{Claim}\label{acht}
For every $C \in \mathcal{C}$, we have $\lambda_{\mathcal{H}}(a,z_C)=8$.
\end{Claim}
\begin{proof}
Clearly, we have $\lambda_{\mathcal{H}}(a,z_C)\leq d_{\mathcal{H}}(z_C)=8$. Next observe that for $i=1,\ldots,4$, the hyperedge $e^*_i$ forms an $az_C$-hyperpath $P_i$. Next, let $\ell_1,\ell_2,\ell_3$ be the literals contained in $C$. For $i=1,2,3$, observe that $av_{\ell_i}v_{\ell_i}^Cz_C$ forms an $az_C$- hyperpath $P_{4+i}$. Finally, observe that there is an $az_C$- hyperpath $P_8$ that can be trimmed to the path $aw_{\ell_1}w_{\ell_1}^Cz_C$. As $P_1,\ldots,P_8$ is a collection of 8 hyperedge-disjoint hyperpaths, for every set $Y \subseteq V(\mathcal{H})$ with $a \in V(\mathcal{H})-Y$ and $z_C \in Y$ and every $i=1,\ldots,8$, there is a hyperedge in $\delta_{\mathcal{H}}(Y)\cap \mathcal{E}(P_i)$, so $d_\mathcal{H}(Y)\geq 8$. Hence the statement follows.
\end{proof}
Let $\vec{\mathcal{H}}$ be a well-balanced orientation of $\mathcal{H}$. By Claim \ref{acht} and as $\vec{\mathcal{H}}$ is well-balanced, we obtain $\lambda_{\vec{\mathcal{H}}}(z_C,a)\geq 4$ for all $C \in \mathcal{C}$. By applying Proposition \ref{propch} to $\vec{\mathcal{H}}-(\{e_1^*,\ldots,e^*_4\}-e^*_i)$, we may suppose that $e_i^*$ is oriented toward $a$ for $i=1,\ldots,4$.  Hence all the hyperedges containing a vertex in $Z$ and some vertices in $W$ are oriented toward the vertex in $Z$. We now define a truth assignment $\phi:X \rightarrow \{TRUE,FALSE\}$ in the following way: For every $x \in X$, we set $\phi(x)=TRUE$ if the hyperedge $\{a,w_x,w_{\bar{x}}\}$ is oriented toward $w_x$ and $\phi(x)=FALSE$ otherwise. In order to show that $\phi$ satisfies $(X,\mathcal{C})$, consider some $C \in \mathcal{C}$, let $\ell_1,\ell_2,\ell_3$ be the literals contained in $C$ and let $Y=z_C\cup \bigcup_{i=1}^3W_{\ell_i}$. As $\vec{\mathcal{H}}$ is well-balanced by Claim \ref{acht}, we have $d^-_{\vec{\mathcal{H}}}(Y)\geq \lambda_{\vec{\mathcal{H}}}(a,z_C)\geq 4$.  As  all the hyperedges containing a vertex in $Z$ and some vertices in $W$ are oriented toward the vertex in $Z$, there exists some $i \in \{1,2,3\}$ such that $\{a,w_{\ell_i},w_{\bar{\ell_i}}\}$ is oriented toward $w_{\ell_i}$, so $\phi(\ell_i)=TRUE$ by construction.

 Hence $\phi$ is a satisfying assignment for $(X,\mathcal{C})$ and so $(X,\mathcal{C})$ is a positive instance of $(3,B2)$-SAT.

Now suppose that $(X,\mathcal{C})$ is a positive instance of $(3,B_2)$-SAT, so there is a satisfying assignment $\phi$ for $(X,\mathcal{C})$. We now create an orientation $\vec{\mathcal{H}}$ of $\mathcal{H}$. First, we orient $e_i^*$ toward $a$ for $i=1,\ldots,4$. Next, we orient  all the hyperedges containing a vertex in $Z$ and some vertices in $W$ toward the vertex in $Z$. Now let $\ell$ be a literal over $X$ that is contained in two clauses $C_1,C_2$ whose order is chosen arbitrarily. We orient the edge $av_\ell$ from $a$ to $v_\ell$. Next, if $\phi(\ell)=TRUE$, for $i=1,2$, we orient $v_\ell v_\ell^{C_i}$ toward $v_\ell^{C_i}$, we orient $w_\ell w_\ell^{C_i}$ toward $w_\ell^{C_i}$ and we give $v_\ell^{C_i}w_\ell^{C_i}$ an arbitrary orientation. If $\phi(\ell)=FALSE$, we orient the edges in $\mathcal{H}[W_\ell]$ so that $v_\ell v_\ell^{C_1}w_\ell^{C_1}w_\ell w_\ell^{C_2} v_\ell^{C_2}v_\ell$ becomes a circuit. Finally, for every $x \in X$, we orient the hyperedge $\{a,w_x,w_{\bar{x}}\}$ toward $w_x$ if $\phi(x)=TRUE$ and toward $w_{\bar{x}}$ if $\phi(x)=FALSE$. This finishes the description of $\vec{\mathcal{H}}$. For an illustration, see Figure \ref{figw2}.

\begin{figure}[h]
    \centering
        \includegraphics[width=\textwidth]{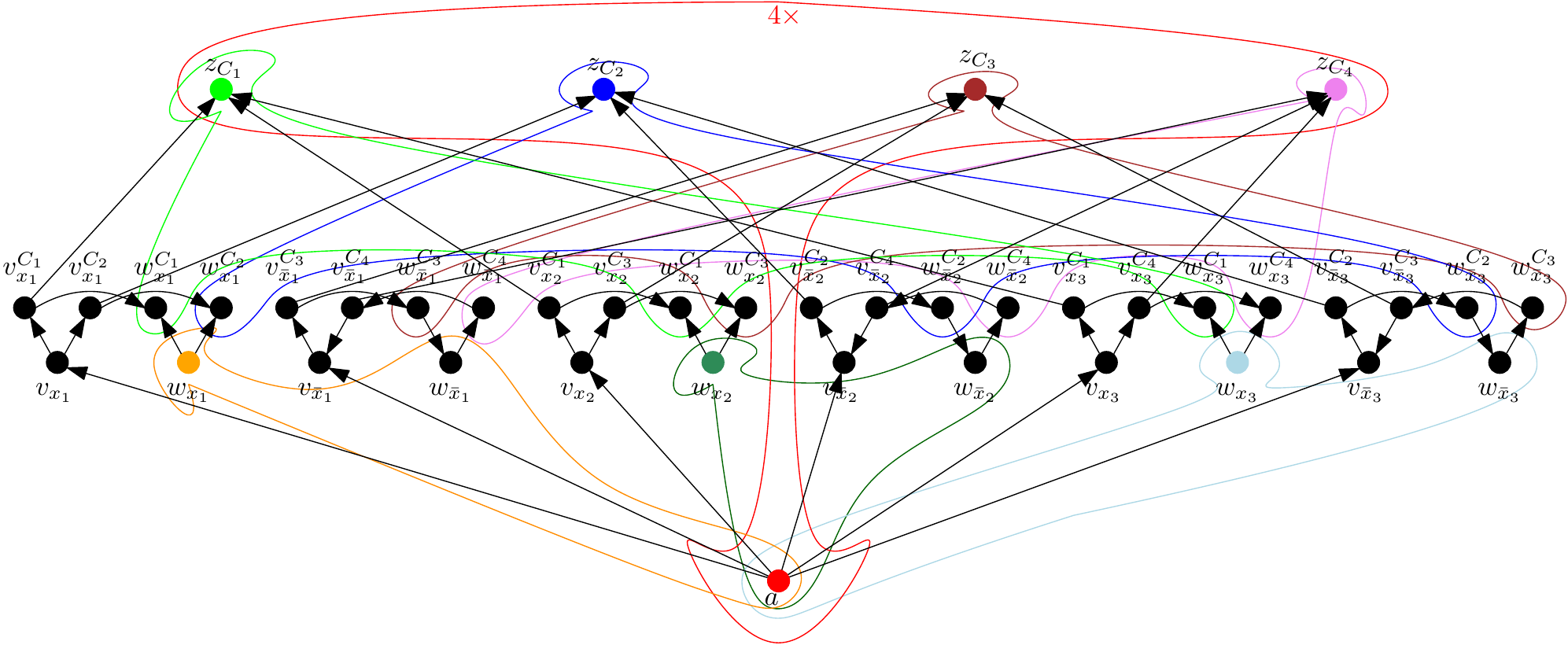}
        \caption{The construction of the orientation $\vec{\mathcal{H}}$ of $\mathcal{H}$ for the instance of Figure \ref{figw1} and the satisfying assignment $\phi$ such that $\phi(x_i)=TRUE$ for $i=1,2,3$. For every dyperedge containing at least three vertices, its head is marked in the color of the dyperedge.}\label{figw2}
\end{figure}

It remains to show that $\vec{\mathcal{H}}$ is well-balanced. First consider some $C \in \mathcal{C}$ and let $\ell_1,\ell_2,\ell_3$ be the literals contained in $C$. As $e_i^*$ is oriented toward $a$ for $i=1,\ldots,4$, we obtain that $\lambda_{\vec{\mathcal{H}}}(z_C,a)\geq 4 = \lfloor \frac{1}{2} d_{\mathcal{H}}(z_C)\rfloor$. Next, for $i=1,2,3$, observe that $\vec{\mathcal{H}}$ contains the directed hyperpath $Q_i=av_{\ell_i}v_{\ell_i}^Cz_C$ if $\phi(\ell_i)=TRUE$ and one of the directed hyperpaths $Q_i=av_{\ell_i}v_{\ell_i}^Cz_C$ and $Q_i=av_{\ell_i}v_{\ell_i}^{C'}w_{\ell_i}^{C'}w_{\ell_i}w_{\ell_i}^{C}v_{\ell_i}^{C}z_C$ where $C'$ is the unique clause distinct from $C$ in which $\ell_i$ is contained in if $\phi(\ell_i)=FALSE$. Finally, as $\phi$ satisfies $C$, there is some $i \in \{1,2,3\}$ such that $\phi(\ell_i)=TRUE$, so the hyperedge $\{a,w_\ell,w_{\bar{\ell}}\}$ is oriented toward $w_\ell$. Now $\vec{\mathcal{H}}$ contains a directed hyperpath that can be trimmed to $aw_{\ell_i}w_{\ell_i}^Cz_C$. As $Q_1,\ldots,Q_4$ are dyperedge-disjoint, we obtain $\lambda_{\vec{\mathcal{H}}}(a,z_C)\geq 4 = \lfloor \frac{1}{2} d_{\mathcal{H}}(z_C)\rfloor$. We next show that $\vec{\mathcal{H}}$ is strongly connected. By the above, we have that $\vec{\mathcal{H}}$ is strongly connected in $Z \cup a$. Let $\ell$ be a literal over $X$ that is contained in the clauses $C_1,C_2$. If $\phi(\ell)=FALSE$, then $\vec{\mathcal{H}}$ contains the directed hyperpath $av_{\ell}v_{\ell}^{C'}w_{\ell}^{C'}w_{\ell}w_{\ell}^{C}v_{\ell}^{C}z_C$ for some ordering $C,C'$ of $C_1,C_2$. If $\phi(\ell)=TRUE$, then $\vec{\mathcal{H}}$ contains directed hyperpaths that can be trimmed to $av_\ell v_\ell^{C_i}z_{C_i}$ and $aw_\ell w_\ell^{C_i}z_{C_i}$, respectively, for $i=1,2$. We obtain that $\vec{\mathcal{H}}$ is strongly connected. This yields $\min\{\lambda_{\vec{\mathcal{H}}}(a,w),\lambda_{\vec{\mathcal{H}}}(w,a)\}\geq 1 = \lfloor \frac{1}{2} d_{\mathcal{H}}(w)\rfloor$ for all $w \in W$. It follows that for any ordered pair $(s_1,s_2)$  in $V(\mathcal{H})$, we have $\lambda_{\vec{\mathcal{H}}}(s_1,s_2)\geq \min\{\lambda_{\vec{\mathcal{H}}}(s_1,a),\lambda_{\vec{\mathcal{H}}}(a,s_2)\}\geq \min\{\lfloor \frac{1}{2} d_{\mathcal{H}}(s_1)\rfloor,\lfloor \frac{1}{2} d_{\mathcal{H}}(s_2)\rfloor\}\geq \lfloor \frac{1}{2} \lambda_{\mathcal{H}}(s_1,s_2)\rfloor$. Hence $\vec{\mathcal{H}}$ is well-balanced, so $\mathcal{H}$ is a positive instance of $WBOH$.

As the size of $\mathcal{H}$ is clearly polynomial in the size of $(X,\mathcal{C})$ and by Theorem \ref{3b2satdure}, the statement follows.
\end{proof}
\section{Conclusion}

We have shown hardness results for the problem of finding Steiner hypertrees in a given hypergraph and three orientations problems in hypergraphs. We further  have shown that two of these problems become  easy when we fix the number of terminals. For these two problems, one may ask whether they are fixed parameter tractable when parameterized by the number of terminals. More concretely, we pose the following problem:

\begin{Problem}\label{prob1}
Is there an algorithm that solves SHT and runs in $O(f(|S|)n^{O(1)})$ for some computable function $f:\mathbb{Z}_{\geq 0}\rightarrow \mathbb{Z}_{\geq 0}$?
\end{Problem}

Observe that due to Lemma \ref{equi}, the analogous problem for SRCOH is equivalent to Problem \ref{prob1}.

For SSCOH even the question whether an analogous result to Theorems \ref{sfix} and \ref{sfix2} exists remains open.

\begin{Problem}
Is there a polynomial time algorithm that solves SSCOH when $|S|$ is fixed?
\end{Problem}

Finally, we could ask whether Theorem \ref{sfix} can be generalized to finding packings of hypertrees.

\begin{Problem}
Can we decide in polynomial time whether a given hypergraph $\mathcal{H}$ contains a packing of $q$ hyperedge-disjoint $S$-Steiner hypertrees when $|S|$ and q are fixed?
\end{Problem}

\end{document}